\documentclass[10pt,a4paper]{article}

\usepackage[final]{optional}

\usepackage{amsfonts, amsmath, wasysym}

\usepackage{amssymb,amsthm,
paralist
}

\usepackage{
latexsym,
}

\usepackage[usenames]{color}

\usepackage{url}

\definecolor{darkgreen}{rgb}{0,0.5,0}
\definecolor{darkred}{rgb}{0.7,0,0}
\usepackage[colorlinks, 
citecolor=darkgreen, linkcolor=darkred
]{hyperref}

\theoremstyle{plain}
\newtheorem{lemma}{Lemma}[section]
\newtheorem{thm}[lemma]{Theorem}
\newtheorem{prop}[lemma]{Proposition}

\theoremstyle{definition}

\newtheorem{rmk}[lemma]{Remark}

\numberwithin{equation}{section}

\newcommand{\m}{\ensuremath{{\cal M}}}
\newcommand{\n}{\ensuremath{{\cal N}}}

\newcommand{\cd}{\ensuremath{{\cal D}}}
\newcommand{\cf}{\ensuremath{{\cal F}}}

\newcommand{\cl}{\ensuremath{{\cal L}}}

\newcommand{\pl}[2]{{\frac{\partial #1}{\partial #2}}}

\newcommand{\plndd}[3]{{\frac{\partial^2 #1}
{{\partial #2}{\partial #3}}}}

\newcommand{\ga}{\gamma}
\newcommand{\Ga}{\Gamma}
\newcommand{\de}{\delta}
\newcommand{\om}{\omega}

\newcommand{\si}{\sigma}

\newcommand{\Si}{\Sigma}

\newcommand{\vph}{\varphi}
\newcommand{\ep}{\varepsilon}

\newcommand{\R}{\ensuremath{{\mathbb R}}}
\newcommand{\N}{\ensuremath{{\mathbb N}}}

\newcommand{\lap}{\Delta}

\newcommand{\grad}{\nabla}

\def\blbox{\quad \vrule height7.5pt width4.17pt depth0pt}

\newcommand{\beq}{\begin{equation}}
\newcommand{\eeq}{\end{equation}}
\newcommand{\beqa}{\begin{equation}\begin{aligned}}
\newcommand{\eeqa}{\end{aligned}\end{equation}}
\newcommand{\brmk}{\begin{rmk}}
\newcommand{\ermk}{\end{rmk}}
\newcommand{\partref}[1]{\hbox{(\csname @roman\endcsname{\ref{#1}})}}
\newcommand{\half}{\frac{1}{2}}

\newcommand{\cmt}[1]{\opt{draft}{\textcolor[rgb]{0.5,0,0}{
$\LHD$ #1 $\RHD$\marginpar{\blbox}}}}

\title{{\sc 
the canonical shrinking soliton \\ associated to a ricci flow
}
\\ 
\cmt{DRAFT with comments}
}
\author{Esther Cabezas-Rivas and Peter M. Topping}
\date{25 November 2009}

\begin{document}

\maketitle
\parskip=10pt

\newcommand{\di}{\partial_i}
\newcommand{\djj}{\partial_j}
\newcommand{\dk}{\partial_k}
\newcommand{\dl}{\partial_l}
\newcommand{\tr}{{\rm tr}}
\newcommand{\lie}{{\cal L}}
\newcommand{\Rop}{{\cal R}}
\newcommand{\Rm}{{\mathrm{Rm}}}
\newcommand{\Ric}{{\mathrm{Ric}}}
\newcommand{\Rc}{{\mathrm{Rc}}}
\newcommand{\RicO}{\overset{\circ}{\Ric}}
\newcommand{\RS}{{\mathrm{R}}}
\newcommand{\Hess}{{\mathrm{Hess}}}
\newcommand{\hess}{{\mathrm{hess}}}
\newcommand{\f}{\ensuremath{{\cal F}}}
\newcommand{\ghat}{\hat{g}}
\newcommand{\fhat}{\hat{f}}
\newcommand{\bop}{+}
\newcommand{\avint}{{\int\!\!\!\!\!\!-}}
\newcommand{\Hbb}{\ensuremath{{\mathbb H}}}

\newcommand{\lexp}{{\cl_{\tau_1,\tau_2}\exp_x}}
\newcommand{\ta}{\ensuremath{{\tau_1}}}
\newcommand{\tb}{\ensuremath{{\tau_2}}}
\newcommand{\lc}{\ensuremath{{\cl Cut}}}
\newcommand{\lcslice}{\ensuremath{{\lc_{\ta,\tb}}}}

\begin{abstract}
To every Ricci flow on a manifold \m\ over a time interval
$I\subset\R_-$, we associate a 
shrinking Ricci soliton on the space-time $\m\times I$.
We relate properties of the original Ricci flow to properties
of the new higher-dimensional Ricci flow equipped with its
own time-parameter.
This geometric construction was discovered by consideration
of the theory of optimal transportation, and in particular
the results of the second author \cite{Lopt}, and McCann
and the second author \cite{MT}; we briefly survey the
link between these subjects.
\end{abstract}

\section{Introduction}
\label{intro}

In 1982, Hamilton \cite{ham3D} introduced the study of 
Ricci flow, which evolves a Riemannian metric $g$ on a manifold 
\m\ under the nonlinear evolution equation
\beq
\label{RFeq}
\pl{g}{t}=-2\,\Ric(g(t)),
\eeq
for $t$ in some time interval $I\subset\R$.
Since then, the subject has developed steadily, and
has become established as an effective bridge between
analysis, geometry and topology (see for example
\cite{P1}, \cite{P2}, \cite{P3} and the overview in
\cite{RFnotes}).

The initial progress relevant to the present paper
was Hamilton's discovery in 1993 of the so-called Harnack 
quantities (see \cite{hamharnack} for more information)
and by 1995, Nolan Wallach \cite[\S 14]{formations}
had proposed that these quantities should arise as some
sort of curvature of some higher-dimensional manifold
or bundle associated to the Ricci flow.
This idea was developed by Chow and Chu \cite{CC1}
who considered the 
space-time manifold 
$\m \times I$, and defined a pair $(\tilde g, \widetilde \nabla)$ 
of a metric on its cotangent bundle \emph{degenerate} in the 
time direction and a $\tilde g$-compatible torsion-free connection 
(which is not unique owing to the degeneracy of $\tilde g$) so that 
the derivatives in the time coordinate direction of the
components of $(\tilde g, \widetilde \nabla)$ resemble
the formulae one can compute for the evolution of the components
of the metric and its Levi-Civita connection under Ricci flow.
(See \cite{CC1} for more details.)
It turns out that Hamilton's matrix Harnack quadratic is 
almost the Riemannian curvature of that space-time 
connection. An improved correspondence is established in 
the work of Chow and Knopf \cite{ChK} by considering
Ricci flow with a `cosmological term'. 
(An example of such a flow would be
$\bar g(\bar t) := \frac1{t} g(t)$, for $\bar t = \log t$, 
where $g(t)$ is a Ricci flow.)

In 2002, Perelman \cite[\S 6]{P1} 
made a new breakthrough along these lines
involving the construction of an essentially Ricci-flat 
manifold of dimension unbounded from above, which we now describe.
The starting point is a Ricci flow which once seen with respect
to a reverse time parameter $\tau:=C-t$ (for some $C\in\R$)
is defined for $\tau$ lying in some interval $I\subset\R_+$.
Let $N\in\N$ be a large natural number, and consider the
manifold $\tilde\m := \m\times I \times S^N$ equipped with 
the metric $\tilde g$ defined by 
$$\tilde g_{ij}=g_{ij}; \quad \tilde g_{00}=\frac{N}{2\tau}+R;
\quad \tilde g_{\alpha\beta}=\tau g_{\alpha\beta},$$
with all remaining metric coefficients $\tilde g_{0i}$, 
$\tilde g_{0\alpha}$ and $\tilde g_{i\alpha}$ equal to zero,
where $i,j$ are coordinate indices on the \m\ factor, 
$\alpha, \beta$ are those on the $S^N$ factor, $0$ represents 
the index of the time coordinate $\tau\in I$, the scalar curvature
is written $R$, and $g_{\alpha\beta}$ is the 
metric on the round $S^N$ of sectional curvature $\frac{1}{2N}$.

The significance of the manifold $(\tilde \m, \tilde g)$ is
that it is \emph{Ricci-flat} up to errors of order $\frac{1}{N}$.
This allowed Perelman to formally apply the Bishop-Gromov
comparison theorem in order to discover his \emph{reduced volume}
\cite{P1}. Perelman's variants of 
Hamilton's Harnack quantities can be recovered
from the full curvature tensor, up to errors
of order $\frac{1}{N}$.

More recently, the theory of 
optimal transportation has been
introduced into the study of Ricci flow, with papers by 
McCann and the second author \cite{MT}, the second author
\cite{Lopt} and then Lott \cite{lottOT}.
We give more details in Section \ref{OTsection},
but for now we mention that a notion of $\cl$-optimal
transportation was introduced in \cite{Lopt} which 
can be used to recover all the important monotonic quantities
for Ricci flow that were discovered by Perelman \cite{P1}
in his analysis of finite-time singularities for Ricci flow
(see \cite{Lopt} and \cite{lottOT}).

The starting point for this paper is the proposal of
John Lott that one might be able to make a formal 
justification of the results in \cite{Lopt} by applying
optimal transportation theory developed for manifolds of
positive Ricci curvature, directly to Perelman's
construction $(\tilde \m,\tilde g)$ (see also \cite{lottOT}).
This seems to be problematic using existing optimal 
transportation theory. However, consideration
of what alternative construction analogous to 
Perelman's $(\tilde \m,\tilde g)$
could lie behind the results
of \cite{Lopt} turns out to be fruitful; in this paper we
are thus led to the following theorem in which we construct 
the \emph{Canonical Shrinking Soliton} 
associated to a Ricci flow on \m.

\begin{thm}
\label{mainthm}
Suppose $g(\tau)$ is a (reverse) Ricci flow -- i.e. a
solution of $\pl{g}{\tau}=2\,\Ric(g(\tau))$ -- defined for
$\tau$ within a time interval $(a,b)\subset (0,\infty)$, on a
manifold $\m$ of dimension $n\in\N$, and with bounded curvature. 
Suppose 
 $N\in\N$ is sufficiently large
to give a positive definite metric $\hat g$ on
 $\hat\m:=\m\times  (a,b)$ defined by
$$\hat g_{ij}=\frac{g_{ij}}{\tau}; 
\qquad 
\hat g_{00}=\frac{N}{2\tau^3}+\frac{R}{\tau}-\frac{n}{2\tau^2};
\qquad
\hat g_{0i}=0,$$
where $i,j$ are coordinate indices on the \m\ factor, 
$0$ represents the index of the time coordinate $\tau\in (a,b)$, 
and the scalar curvature of $g$ is written as $R$.

Then up to errors of order $\frac{1}{N}$, the metric $\hat g$
is a gradient shrinking Ricci soliton on the higher dimensional
space $\hat\m$:
\beq
\label{solitonequation}
\Ric(\hat g)+\Hess_{\hat g} \left(\frac{N}{2\tau}\right)
\simeq\half \hat g,
\eeq
by which we mean that 
 the quantity 
$$N\left[\Ric(\hat g)+\Hess_{\hat g}\left(\frac{N}{2\tau}\right)
- \half \hat g\right]$$ 
is locally bounded independently of $N$, with respect to any fixed
metric on $\hat\m$.
\end{thm}

One aspect of this theorem is that given
a Ricci flow over the time interval $(a,b)$, it is then 
natural to introduce an additional time parameter
and consider Ricci flow using $\m\times (a,b)$
as the underlying manifold.

This theorem serves as an example of an application of the
theory of optimal transportation to prove a new result in
a different field. The route between the theorems of optimal
transportation and this construction is described in Sections
\ref{OTsection} and \ref{heuristicssect}.
It will become apparent, from Section \ref{heuristicssect}
and \cite{Lopt}, that our Canonical Shrinking Soliton encodes
various monotonic quantities which underpin Perelman's
work on Ricci flow \cite{P1,P2,P3}, 
including entropies and quantities involving $\cl$-length
(see Section \ref{OTsection}). 
In Section \ref{steadysect} we describe a steady 
space-time soliton
construction which will generate other monotonic quantities.

\brmk
Our construction also encodes Perelman's variants of 
Hamilton's Harnack quantities.
More precisely, the components of the $(4,0)$ curvature tensor 
$\tau Rm( \hat g)$ coincide (up to errors of order $\frac1{N}$) 
with the components of Perelman's version of the
matrix Harnack expression \cite{hamharnack, P1}.
In \cite{CT2} we will explain how a variant of these
ideas can be used to prove new Harnack inequalities.
\ermk

\brmk
Although each side of \eqref{solitonequation} evaluated
on the pair $(\pl{}{\tau},\pl{}{\tau})$ will have magnitude
of order $N$, the theorem tells us that their difference 
does not even have any terms of order
$1$, only those of order $\frac{1}{N}$.
All errors disappear when the original Ricci flow $g(\tau)$
is a homothetically shrinking Einstein manifold, shrinking
to nothing at $\tau=0$.
\ermk

\emph{Acknowledgements:} We thank Robert McCann for discussions
relating to an observation of Tom Ilmanen used in Section
\ref{OTsection}. Part of this work was carried out at the 
Centre de Recerca Matem\`atica, Barcelona, and the Institut
Henri Poincar\'e, Paris, and we thank these institutions
for their hospitality. This work was partly supported by
the Leverhulme Trust. The first author was partly 
supported  by DGI(Spain) and FEDER Project MTM2007-65852.

\section{The calculations}

In this section, we give exact formulae for the 
Christoffel symbols of $\hat g$ from Theorem \ref{mainthm}, 
and approximate 
formulae for its Ricci curvatures and for the Hessian of 
$\frac{N}{2\tau}$ with respect to $\hat g$. 
Theorem \ref{mainthm} will then 
be seen to follow easily.

\begin{prop}
In the setting of Theorem \ref{mainthm}, if $\Ga^i_{jk}$
are the Christoffel symbols of $g(\tau)$ at some point
$x\in\m$, then the Christoffel symbols of $\hat g$
at $(x,\tau)$ are given by
$$
\hat \Ga^i_{jk}=\Ga^i_{jk};\quad
\hat \Ga^i_{j0}={R^i}_j-\frac{{\de^i}_j}{2\tau};\quad
\hat \Ga^i_{00}=-\half g^{ij}\pl{R}{x^j};\quad
$$
$$
\hat \Ga^0_{jk}=\hat g_{00}^{-1}
\left(\frac{g_{jk}}{2\tau^2}-\frac{R_{jk}}{\tau}\right);\quad
\hat \Ga^0_{i0}=\frac{1}{2\tau} \hat g_{00}^{-1}\pl{R}{x^i};\quad
\hat \Ga^0_{00}=\half \hat g_{00}^{-1}
\left[-\frac{3N}{2\tau^4}-\frac{R}{\tau^2}+\frac{R_\tau}{\tau}
+\frac{n}{\tau^3}\right]
$$
\end{prop}
This is a straightforward computation from the definition of the
Christoffel symbols
$$\hat\Ga^a_{bc}:=\half \hat g^{ad}\left(
\pl{\hat g_{cd}}{x^b}+\pl{\hat g_{bd}}{x^c}-\pl{\hat g_{bc}}{x^d}
\right),$$
where $a, b, c, d$ are arbitrary indices,
and the equation of Ricci flow. Using the
standard formula for the coefficients of the Ricci curvature
$$\hat R_{ab}=\pl{\hat\Ga^c_{ab}}{x^c} -\pl{\hat\Ga^c_{ac}}{x^b}
+\hat\Ga^c_{ab}\hat\Ga^d_{cd}-\hat\Ga^d_{ac}\hat\Ga^c_{bd},$$
the formula for the coefficients of $\Hess_{\hat g}(f)$
$$\hat\grad^2_{ab}(f)=\plndd{f}{x^a}{x^b}-\pl{f}{x^c}\hat\Ga^c_{ab},$$
the equation for the evolution of $R$
$$R_\tau+\lap R+2|\Ric|^2=0,$$
(see for example \cite[Proposition 2.5.4]{RFnotes})
and the contracted Bianchi identity
$$\grad_i {R^i}_j=\half \grad_j R,$$
one readily verifies the following:

\begin{prop}
\label{ric_hess_prop}
Fixing $\tau>0$, a time at which the Ricci flow exists, and
fixing local coordinates $\{x^i\}$ in a neighbourhood $U$ of 
some $p\in\m$, then in any neighbourhood 
$V\subset\subset U\times (a,b)$ of $(p,\tau)$, we have
$$\hat R_{ij}\simeq R_{ij};\qquad
\hat R_{i0}\simeq -\half \grad_i R;\qquad
\hat R_{00}\simeq -\frac{R_\tau}{2}-\frac{R}{2\tau},$$
where $\simeq$ denotes equality of the coefficients
up to an error bounded in magnitude by $\frac{C}{N}$, 
with $C>0$ a constant independent of $N$ 
(but depending on $V$ and the choice of coordinates).
Moreover,  
we have
$$\hat\grad^2_{ij}({\textstyle \frac{N}{2\tau}})
\simeq \frac{g_{ij}}{2\tau}-R_{ij};\qquad
\hat\grad^2_{i0}({\textstyle \frac{N}{2\tau}})
\simeq \frac{\grad_i R}{2};\qquad
\hat\grad^2_{00}({\textstyle \frac{N}{2\tau}})
\simeq \frac{N}{4\tau^3}+\frac{R}{\tau}
-\frac{n}{4\tau^2}+\frac{R_\tau}{2}.$$
\end{prop}

By combining the formulae of Proposition \ref{ric_hess_prop}
and the definition of $\hat g$, we deduce Theorem \ref{mainthm}.

\section{Optimal Transportation on Ricci flows}
\label{OTsection}

Suppose that $(\m,g)$ is a closed (compact, no boundary) 
Riemannian manifold, and $\nu_1$ and $\nu_2$ are two
Borel probability measures on \m. 
For $p\in [1,\infty)$,
we define 
the $p$-Wasserstein distance 
$W_p$ between $\nu_1$ and $\nu_2$ to be
\beq
\label{Wpdef}
W_p^{g}(\nu_1,\nu_2):=\left[\inf_{\pi\in\Ga(\nu_1,\nu_2)}
\int_{\m\times\m} d^p(x,y) d\pi(x,y)\right]^\frac{1}{p},
\eeq
where $d(\cdot,\cdot)$ is the Riemannian distance function
induced by $g$
and $\Ga(\nu_1,\nu_2)$ is the space of Borel probability 
measures on $\m\times\m$ with marginals $\nu_1$ and $\nu_2$.

A basic principle in the subject -- see Sturm and von Renesse 
\cite{SvR} and the references therein --
is that two probability measures evolving under an
appropriate diffusion
equation should get closer in the Wasserstein sense provided
the manifold satisfies some curvature condition, most famously
positive Ricci curvature. 

In \cite{MT}, McCann and the second author showed that this
type of contractivity 
on an \emph{evolving} manifold
$(\m,g(\tau))$ \emph{characterises} super-solutions of the Ricci flow
(parametrised backwards in time) by which we mean solutions to
\beq
\label{superRF}
\pl{g}{\tau}\leq 2\,\Ric(g(\tau)).
\eeq
Here the relevant notion of diffusion on an evolving
manifold $(\m,g(\tau))$ moves the probability density
$u$ (with respect to the evolving Riemannian measure $\mu_{g(\tau)}$)
by the parabolic equation
$$\pl{u}{\tau}=\lap_{g(\tau)} u - 
\half\tr\left(\pl{g}{\tau}\right)u,$$
and we refer to the resulting one-parameter families of 
measures simply as \emph{diffusions}.
This way, if we define a local top-dimensional form
$\om(\tau):=u\,dV_{g(\tau)}$, then 
\beq
\label{omevol}
\pl{\om}{\tau}=\lap_{g(\tau)}\om,
\eeq
where $\lap$ is here the connection Laplacian,
and thus the evolution of the measures corresponds 
to Brownian motion.
The following characterisation was proved for the $W_2$ distance
in \cite{MT}. Tom Ilmanen has pointed out to us (via Robert McCann)
that this extends to the case of $W_1$ distance, giving:

\begin{thm} (cf. \cite[Theorem 2]{MT})
\label{W1thm}
Suppose that $\m$ is a closed manifold equipped with 
a smooth family of metrics $g(\tau)$ for 
$\tau\in [\tau_1,\tau_2]\subset \R$.
Then the following are equivalent:
\begin{enumerate}[($A$)]
\item
$g(\tau)$ is a super Ricci flow (i.e. satisfies \eqref{superRF});
\item
whenever $\tau_1<a<b<\tau_2$ and 
$\nu_1(\tau)$, $\nu_2(\tau)$ are 
diffusions (as defined above) for $\tau\in (a,b)$,
the function $\tau\mapsto W_1^{g(\tau)}(\nu_1(\tau),\nu_2(\tau))$ 
is weakly decreasing
in $\tau\in (a,b)$;
\item
whenever $\tau_1<a<b<\tau_2$ and $f:\m\times (a,b)\to\R$
is a solution to 
$-\frac{\partial f}{\partial \tau}=\Delta_{g(\tau)}f$, 
the 
Lipschitz constant of $f(\cdot,\tau)$ with respect to $g(\tau)$
is weakly increasing in $\tau$.
\end{enumerate}
\end{thm}

\begin{proof}
All implications in this theorem are proved exactly as in
\cite{MT} except for  $(A)\implies (B)$
whose proof, pointed out by Ilmanen, we now describe.
Suppose that $g(\tau)$ is a super Ricci flow, and 
that $\nu_1(\tau)$ and $\nu_2(\tau)$ are two diffusions,
all defined for $\tau$ in a neighbourhood of $\tau_0\in (a,b)$.
By Kantorovich-Rubinstein duality (see e.g. \cite[\S 1.2.1]{villani})
we have
\beq
\label{KRduality}
\begin{aligned}
W_1^{g(\tau)}(\nu_1(\tau),\nu_2(\tau))&:=\max\bigg\{
\int_\m \vph d\nu_1(\tau) - \int_\m \vph d\nu_2(\tau)\ \bigg|\ 
\vph:\m\to\R \text{ is Lipschitz }\\
&\qquad\qquad\qquad\qquad\qquad \text{ and }
\|\vph\|_{Lip}\leq 1 \text{ with respect to }g(\tau)\bigg\}.
\end{aligned}
\eeq
Let $\vph_0:\m\to\R$ be a function which achieves the
maximum in this variational problem at time $\tau_0$, 
and extend $\vph_0$
to a function $\vph:\m\times [\tau_0-\ep,\tau_0]\to\R$
for some $\ep>0$ by solving the equation
$$-\pl{\vph}{\tau}=\lap_{g(\tau)}\vph.$$
By the implication $(A)\implies (C)$ of the theorem (proved in
\cite{MT}) for all $\tau\in [\tau_0-\ep,\tau_0]$ we have
$\|\vph(\cdot,\tau)\|_{Lip}\leq 1$ and therefore $\vph(\cdot,\tau)$
can be used as a competitor in the variational problem
\eqref{KRduality} to see that
$$W_1^{g(\tau)}(\nu_1(\tau),\nu_2(\tau))\geq
\int_\m \vph(\cdot,\tau) d\nu_1(\tau) 
- \int_\m \vph(\cdot,\tau) d\nu_2(\tau).$$
But by an integration by parts formula \cite[\S 6.3]{RFnotes}
we know that the functions
$$\tau\mapsto \int_\m \vph(\cdot,\tau) d\nu_1(\tau)\qquad
\text{and}\qquad
\tau\mapsto \int_\m \vph(\cdot,\tau) d\nu_2(\tau)$$
are each independent of $\tau$, so we deduce that
$$W_1^{g(\tau)}(\nu_1(\tau),\nu_2(\tau))\geq
W_1^{g(\tau_0)}(\nu_1(\tau_0),\nu_2(\tau_0)).$$
\end{proof}

In the next section, we will demonstrate how the manifold
$(\hat\m,\hat g)$ of Theorem \ref{mainthm} 
arises naturally by trying to reconcile Theorem
\ref{W1thm} with 
the main result proved by the second author in 
\cite{Lopt}, which we now rephrase into the most suggestive 
form for our present purposes.
The idea of $\cl$-optimal transportation \cite{Lopt} is
to transport a probability measure from one time slice
of a Ricci flow to another, using a cost function derived
from Perelman's $\cl$-length.
More precisely, given a time interval 
$[\tau_1,\tau_2]\subset (0,\infty)$ in the domain of definition
of the Ricci flow, we consider the cost function 
$c:\m\times\m\to\R$ induced by the Lagrangian 
$L(x,v,\tau):=\sqrt{\tau}(R(x,\tau)+|v|^2-\frac{n}{2\tau})$
which gives 
$$c(x,y)=\inf_\ga \int_{\tau_1}^{\tau_2} 
\sqrt{\tau}\left(R(\ga(\tau),\tau)+|\ga '(\tau)|^2-\frac{n}{2\tau}
\right)d\tau,$$
where the infimum is taken over all $C^1$ curves
$\ga:[\tau_1,\tau_2]\to\m$ for which $\ga(\tau_1)=x$
and $\ga(\tau_2)=y$.
Using the definitions from \cite[\S 7]{P1} and \cite{Lopt}
$$\cl(\ga)=\int_{\tau_1}^{\tau_2} 
\sqrt{\tau}(R(\ga(\tau),\tau)+|\ga '(\tau)|^2)d\tau;
\qquad Q(x,\tau_1;y,\tau_2)=\inf_\ga \cl(\ga),$$
where the infimum is again over curves $\ga$ as above,
we can write
$$c(x,y)=Q(x,\tau_1;y,\tau_2)-n(\sqrt{\tau_2} - \sqrt{\tau_1}).$$
This cost function then induces a distance from one
Borel probability measure $\nu_1$ (viewed as existing at time
$\tau_1$) to another $\nu_2$ (viewed at time $\tau_2$) via 
the formula
\beq
\begin{aligned}
\cd(\nu_1,\tau_1;\nu_2,\tau_2)&=
\inf_{\pi\in\Ga(\nu_1,\nu_2)}\int_{\m\times\m}Q(x,\tau_1;y,\tau_2)
d\pi(x,y) - n(\sqrt{\tau_2} - \sqrt{\tau_1})\\
&=: V(\nu_1,\tau_1;\nu_2,\tau_2) - n(\sqrt{\tau_2} - \sqrt{\tau_1}),
\end{aligned}
\eeq
using $V$ as defined in \cite{Lopt}.
From the following theorem, one can recover \cite{Lopt} most of
Perelman's monotonic quantities (both involving
entropies and $\cl$-length) which are central in his 
work on Ricci flow \cite{P1,P2,P3}.

\begin{thm} (Equivalent to \cite[Theorem 1.1]{Lopt}.)
\label{loptthm}
Suppose that $g(\tau)$ is a Ricci flow on a closed manifold
$\m$ over an open time interval containing 
$[\bar{\tau}_1,\bar{\tau}_2]$,
and suppose that $\nu_1(\tau)$ and $\nu_2(\tau)$ are two 
diffusions (in the same sense as in Theorem \ref{W1thm}) 
defined for $\tau$ in neighbourhoods
of $\bar{\tau}_1$ and $\bar{\tau}_2$ respectively.
Then the distance between the diffusions decays
in the sense that for $s\geq 1$ sufficiently close to $1$,
$$\cd(\nu_1(s\bar\tau_1),s\bar\tau_1;\nu_2(s\bar\tau_2),s\bar\tau_2)
\leq s^{-\half}
\cd(\nu_1(\bar{\tau}_1),\bar{\tau}_1;
\nu_2(\bar{\tau}_2),\bar{\tau}_2).$$
\end{thm}

This formulation of the theorem indicated to us that we should
look for a context in which the result arises as an application
of Theorem \ref{W1thm}.
This leads to Theorem \ref{mainthm},
and we explain the connection in the next section.

\section{The relationship between $(\hat\m,\hat g)$ and
$\cl$-optimal transportation}
\label{heuristicssect}

In this section, as opposed to all others 
in this paper,
we allow ourselves to make purely heuristic arguments.
We present a formal argument to recover Theorem \ref{loptthm}
by applying Theorem \ref{W1thm}  
to a flow starting at
the specific manifold $(\hat\m,\hat g)$ of Theorem
\ref{mainthm}.
While this is the simplest way to explain the link between
our new manifold $(\hat\m,\hat g)$ and optimal transportation,
we stress that the main interest here is the reverse flow
of ideas: the Ricci soliton $(\hat\m,\hat g)$ was 
\emph{discovered} by trying to find a context in which
Theorem \ref{W1thm} formally implied Theorem \ref{loptthm}.

We begin this section by arguing formally that
shortest paths in $(\hat\m,\hat g)$ correspond to $\cl$-geodesics
for the original Ricci flow.
Suppose $x,y\in\m$ and $[\tau_1,\tau_2]\subset (0,\infty)$
lies within the time domain on which the Ricci flow is defined.
Consider paths $\Ga:[\tau_1,\tau_2]\to\hat\m$ connecting
$(x,\tau_1)$ and $(y,\tau_2)$ in $\hat\m$ of the form
$\Ga(\tau)=(\ga(\tau),\tau)$, where $\ga:[\tau_1,\tau_2]\to\m$
satisfies $\ga(\tau_1)=x$ and $\ga(\tau_2)=y$.
Then 
\beq
\label{Gaexp}
\begin{aligned}
Length(\Ga)&=\int_{\tau_1}^{\tau_2} \left|\ga'(\tau)+\pl{}{\tau}
\right|_{\hat g}d\tau\\
&= \int_{\tau_1}^{\tau_2} \left[ 
\frac{|\ga'|_{g(\tau)}^2}{\tau}
+\frac{N}{2\tau^3}+\frac{R}{\tau}-\frac{n}{2\tau^2}\right]^\half
d\tau \\
&=\int_{\tau_1}^{\tau_2} \left[ \frac{N}{2\tau^3}\right]^\half
\left(1+\frac{\tau^2|\ga'|^2}{N}+\frac{\tau^2 R}{N}
-\frac{\tau n}{2N} + O({\textstyle\frac{1}{N^2}})\right)d\tau\\
&=\sqrt{2N}(\tau_1^{-\half}-\tau_2^{-\half})
+\frac{1}{\sqrt{2N}}\left[\cl(\ga)
-n(\sqrt{\tau_2}-\sqrt{\tau_1})\right] 
+ O({\textstyle\frac{1}{N^{3/2}}}).
\end{aligned}
\eeq
Just as in \cite[\S 6]{P1}, this indicates that minimising
paths $\Ga$ should essentially arise from
$\cl$-geodesics $\ga$ for large $N$, and suggests the formula
\beq
\label{distexpansion}
d_{\hat g}((x,\tau_1),(y,\tau_2))=
\sqrt{2N}(\tau_1^{-\half}-\tau_2^{-\half})
+\frac{1}{\sqrt{2N}}\left[Q(x,\tau_1;y,\tau_2)
-n(\sqrt{\tau_2}-\sqrt{\tau_1})\right] 
+ O({\textstyle\frac{1}{N^{3/2}}}).
\eeq

We now wish to apply Theorem \ref{W1thm} to certain
diffusions on a reverse Ricci flow $G(s)$ 
on \emph{space-time}
 starting at time $s=1$ with the metric $G(1)=\hat g$
on $\hat\m=\m\times (a,b)$.
Since $\hat g$ satisfies the (approximate) Ricci soliton
equation \eqref{solitonequation}, the 
theory of Ricci solitons (see \cite[\S 1.2.2]{RFnotes})
tells us that (modulo errors) we can take 
$G(s)$ to be
\beq
\label{Gdef}
G(s):=s\,\psi_s^*(\hat g)
\eeq 
where $\psi_1:\hat\m\to\hat\m$ is the identity,
and $\psi_s$ is the family of maps 
obtained by integrating the vector field
$X_s:=-\frac{1}{s}\hat\grad (\frac{N}{2\tau})$ 
with $\hat\grad$ representing
the gradient with respect to $\hat g$.

We may compute
\beq
\label{Xapprox}
X_s=\frac{N}{2s\tau^2}\hat g_{00}^{-1}\pl{}{\tau}=
\frac{\tau}{s}\pl{}{\tau}+O({\textstyle \frac{1}{N}}),
\eeq
and by neglecting the error of order $\frac{1}{N}$ (but see
also Section \ref{meansect}) this integrates to
$$\psi_s(x,\tau)\simeq (x,s\tau).$$
In particular, the (approximate, reverse) Ricci flow $G(s)$ operates
by pulling back $\hat g$ in time, and scaling appropriately.

Applying Theorem \ref{W1thm}, we find that two diffusions
on the evolving manifold with metric
$G(s)$ should get closer in the $W_1$ sense
as $s$ increases. 
The Ricci flow $g(\tau)$ we use to generate $(\hat\m,\hat g)$
will be that in the hypotheses of
Theorem \ref{loptthm}. 
The measures we wish
to put into Theorem \ref{W1thm} will be derived from
the measures $\nu_1(\tau)$ and $\nu_2(\tau)$ appearing in the
hypotheses of Theorem \ref{loptthm}, together with the
corresponding $\bar\tau_1$ and $\bar\tau_2$.
Let $F_\tau:\m\to\hat\m$ be defined to be the embedding
$F_\tau(x)=(x,\tau)$; then the initial measures 
we wish to put into Theorem \ref{W1thm} are 
$(F_{\bar\tau_k})_\# \nu_k(\bar\tau_k)$, for $k=1,2$,
which are each supported on a time-slice in $\hat\m$.
Because of the extreme stretching of the $\tau$ direction
in the metric $\hat g$ (recall that $\hat g_{00}$ is of 
order $N$) there is essentially no diffusion of the
measures in the $\tau$ direction, and we view them as
remaining supported in the time slices $\m\times \{\bar\tau_k\}$.
They then evolve mainly under diffusion in the $\m$ factor,
under the Laplacian induced by $G_{ij}(s)$. By \eqref{Gdef}
and the definition of $\hat g$ from Theorem \ref{mainthm}, 
on the time-slice $\m\times \{\bar\tau_k\}$ we have (approximately)
$$G_{ij}(s)= s [\psi_s^*(\hat g)]_{ij}
= s [ \hat g|_{\m\times \{s\bar\tau_k\}}]_{ij}
=\frac{g_{ij}(s\bar\tau_k)}{\bar\tau_k}.$$
Therefore, the measures evolve by diffusion in the $\m$
factor under 
$$\lap_{\frac{g(s\bar\tau_k)}{\bar\tau_k}}
=\bar\tau_k \lap_{g(s\bar\tau_k)},$$ 
which is also 
the evolution of $s\mapsto \nu_k(s\bar\tau_k)$.
In other words, we have (formally) deduced that
$$s\mapsto W_1^{G(s)}((F_{\bar\tau_1})_\# \nu_1(s\bar\tau_1),
(F_{\bar\tau_2})_\# \nu_2(s\bar\tau_2))$$
is weakly decreasing.

If we now push forward this whole construction under
the maps $\psi_s$, and adopt the abbreviation
$$\hat\nu_k(\tau):=(F_{\tau})_\# \nu_k(\tau),$$
then we find that 
\beq
\label{W1mono}
s\mapsto W_1^{s\hat g}(\hat\nu_1(s\bar\tau_1),
\hat\nu_2(s\bar\tau_2))=
s^\half W_1^{\hat g}(\hat\nu_1(s\bar\tau_1),
\hat\nu_2(s\bar\tau_2))
\eeq
is weakly decreasing.

Now the expansion \eqref{distexpansion} of the Riemannian
distance on $(\hat\m,\hat g)$ suggests that
$$W_1^{\hat g}(\hat\nu_1(\tau_1),\hat\nu_2(\tau_2))\simeq
\sqrt{2N}(\tau_1^{-\half}-\tau_2^{-\half})
+\frac{1}{\sqrt{2N}}
\cd (\nu_1(\tau_1),\tau_1;\nu_2(\tau_2),\tau_2)
.$$
Therefore, the monotonicity in \eqref{W1mono} implies that
$$s\mapsto\sqrt{2N}(\bar\tau_1^{-\half}-\bar\tau_2^{-\half})
+\frac{s^\half}{\sqrt{2N}}
\cd (\nu_1(s\bar\tau_1),s\bar\tau_1;\nu_2(s\bar\tau_2),s\bar\tau_2)
$$
is monotonically decreasing, and hence so is
$$s\mapsto s^\half 
\cd (\nu_1(s\bar\tau_1),s\bar\tau_1;\nu_2(s\bar\tau_2),s\bar\tau_2),$$
which is the content of Theorem \ref{loptthm} as desired.

\section{Construction of a space-time steady soliton}
\label{steadysect}

In this section we make a steady soliton construction
analogous to the shrinking soliton construction of Theorem
\ref{mainthm}. There is also an expanding soliton construction
similar to that of Theorem \ref{mainthm} which we will use
in \cite{CT2} to prove Harnack inequalities.
Only the shrinking case is adapted to Perelman's $\cl$-length
and his monotonic quantities which are so important in
the study of finite-time singularities \cite{P1, P2, P3}. 
However, the steady case is the simplest construction of all
and has its own potential applications.

\begin{thm}
\label{steadythm}
Suppose $g(\tau)$ is a (reverse) Ricci flow -- i.e. a
solution of $\pl{g}{\tau}=2\,\Ric(g(\tau))$ -- defined for
$\tau$ within a time interval $(a,b)\subset \R$, on a
manifold $\m$ of dimension $n\in\N$, and with bounded curvature.
Suppose 
$N\in\N$ is sufficiently large
to give a positive definite metric $\bar g$ on 
$\hat\m:=\m\times  (a,b)$ defined by
$$\bar g_{ij}=g_{ij}; 
\qquad 
\bar g_{00}=N+R;
\qquad
\bar g_{0i}=0,$$
where $i,j$ are coordinate indices on the \m\ factor and  
$0$ represents the index of the time coordinate $\tau\in (a,b)$.

Then up to errors of order $\frac{1}{N}$, the metric $\bar g$
is a gradient steady Ricci soliton on $\hat\m$:
\beq
\label{steadysolitonequation}
\Ric(\bar g)+\Hess_{\bar g} \left(-N\tau\right)
\simeq 0,
\eeq
in the same sense as in Theorem \ref{mainthm}.
\end{thm}

For this metric, 
the Christoffel symbols can be computed to be 
$$
\bar \Ga^i_{jk}=\Ga^i_{jk};\quad
\bar \Ga^i_{j0}={R^i}_j;\quad
\bar \Ga^i_{00}=-\half g^{ik}\pl{R}{x^k};\quad
$$
$$
\bar \Ga^0_{jk}=-\frac{1}{N+R}R_{jk};\quad
\bar \Ga^0_{j0}=\frac{1}{2} \pl{}{x^j}\ln(N+R);\quad
\bar \Ga^0_{00}=\half \pl{}{\tau} \ln (N+R).
$$
The Ricci coefficients are, up to errors of order $\frac{1}{N}$,
$$\bar R_{ij}\simeq R_{ij};\qquad
\bar R_{i0}\simeq - \frac{\grad_i R}{2};\qquad
\bar R_{00}\simeq -\frac{R_\tau}{2},$$
and the coefficients of $\Hess_{\bar g}(-N\tau)$
are
$$\bar\grad^2_{ij}(-N\tau)
\simeq -R_{ij};\qquad
\bar\grad^2_{i0}(-N\tau)
\simeq \frac{\grad_i R}{2};\qquad
\bar\grad^2_{00}(-N\tau)
\simeq \frac{R_\tau}{2},$$
which yields Theorem \ref{steadythm}.

Just as the $(\hat\m,\hat g)$ of Theorem \ref{mainthm} 
encodes Perelman's $\cl$-length in its geodesic distance,
in the sense of \eqref{Gaexp}, the manifold
$(\hat\m,\bar g)$ of Theorem \ref{steadythm} 
encodes Li-Yau's length \cite{LY}
$$\cl_0(\ga):=\int_{\tau_1}^{\tau_2} 
(R(\ga(\tau),\tau)+|\ga '(\tau)|^2)d\tau$$
which predates $\cl(\ga)$, 
via the formula (using the same notation as in
\eqref{Gaexp})
$$Length(\Ga)=
\sqrt{N}(\tau_2-\tau_1)
+\frac{1}{2\sqrt{N}}\cl_0(\ga)
+ O({\textstyle\frac{1}{N^{3/2}}}).$$

If one considers optimal transportation on this 
alternative $(\hat\m,\bar g)$, then one recovers the
variant of $\cl$-optimal transportation using the 
$\cl_0$-length of Li-Yau, as studied in
\cite{lottOT}.

\section{Mean curvature of time-slices}
\label{meansect}

In the framework of Theorem \ref{mainthm}, the underlying manifold 
${\mathcal M}$ of our original Ricci flow can be regarded as a
hypersurface $\m\times\{\tau\}$ of the ambient
$(\hat {\mathcal M}, \hat g)$. 
Its mean curvature is
$H
=H^{\hat g}_{\m\times\{\tau\}}
 = \hat g_{00}^{-\frac1{2}} (R - \frac{n}{2 \tau})$; 
in fact, the mean curvature vector 
$\vec H^{\hat g}_{\m\times\{\tau\}} 
:= - H \nu = - H \hat g_{00}^{-\frac1{2}} 
\frac{\partial}{\partial \tau}$ gives precisely the term we 
neglected in the approximate computation of $X_s$ from
\eqref{Xapprox}.
More precisely, we have
the exact formula
\beq
\label{exactformula}
-\hat\grad \left(\frac{N}{2\tau}\right) 
= \tau \frac{\partial}{\partial \tau} 
+ \vec H^{\hat g}_{\m\times\{\tau\}}.
\eeq
If we suppose that $\tau_0\in U\subset\subset (a,b)$, we can pick
$\ep>0$ sufficiently small so that for $s\in (1-\ep,1+\ep)$
the maps $\psi_s$ arising from
integrating the vector field 
$$X_s:=-\frac{1}{s}\hat\grad(\frac{N}{2\tau})$$
starting with $\psi_1:\hat\m\to\hat\m$ the identity, 
restrict to well-defined maps
$\m\times U \mapsto \hat\m$, diffeomorphic onto their images. 
It then makes sense at a rigorous level to define
$G(s):=s\,\psi_s^*(\hat g)$ as a flow on $\m\times U$
for $s\in (1-\ep,1+\ep)$,
which is then approximately a reverse Ricci flow as in Section
\ref{heuristicssect}. By reducing $\ep>0$ further, we can
also be sure that $\m\times\{s\tau_0\}\subset \psi_s(\m\times U)$
for $s\in (1-\ep,1+\ep)$.
 
We then have
the following exact statement (i.e. involving no errors at
all) relating reverse mean curvature flow to certain one-parameter
families of time-slices $\m\times\{\tau\}$, which is in a 
similar spirit to \cite[Lemma 3.2]{CC2}.
\begin{thm}
Given $U$, $\tau_0$, $\ep$, $\psi_s$ and the flow $G(s)$ as above, 
if we define 
$\f:\m\times (1-\ep,1+\ep)\to \m\times U$ 
by $\cf(x,s)=\psi_s^{-1}(x,s\tau_0)$,
then the one-parameter family of submanifolds
$$\n_s:=\cf(\m\times\{s\})=\psi_s^{-1}(\m\times\{s\tau_0\})$$
is a reverse mean curvature flow within the flow
$G(s)$ 
in the sense that 
$$-\pl{\cf(x,s)}{s}=\vec H^{G(s)}_{\n_s}(\cf(x,s)).$$
\end{thm}
Note that here both the (approximate) Ricci flow $G(s)$
and the mean curvature flow $\n_s$ evolve in the same 
direction -- backwards with respect to $s$.

\begin{proof}
By differentiating the expression $\psi_s \circ \psi_s^{-1}=identity$
with respect to $s$, we find that
\beq
\begin{aligned}
(\psi_s)_*\pl{\psi_s^{-1}}{s}(x,s\tau_0)
&=-\pl{\psi_s}{s}(\psi_s^{-1}(x,s\tau_0)) 
=-X_s(x,s\tau_0)
=\frac{1}{s}\hat\grad \left(\frac{N}{2\tau}\right) (x,s\tau_0)\\
&=-\tau_0\pl{}{\tau}
-\frac{1}{s}\vec H^{\hat g}_{\m\times\{s\tau_0\}}(x,s\tau_0),
\end{aligned}
\eeq
where we have used \eqref{exactformula}.
Therefore, we may compute
\begin{equation*}
\begin{aligned}
\pl{}{s}\left(\psi_s^{-1}(x,s\tau_0)\right)
&= \pl{\psi_s^{-1}}{s}(x,s\tau_0) 
+ (\psi_s^{-1})_*\left(\tau_0\pl{}{\tau}\right)
= - (\psi_s^{-1})_*
\left[\frac{1}{s}\vec H^{\hat g}_{\m\times\{s\tau_0\}}(x,s\tau_0)\right]\\
&=-\frac{1}{s}\vec H^{\psi_s^*(\hat g)}_{\n_s}(\psi_s^{-1}(x,s\tau_0))
=-\vec H^{G(s)}_{\n_s}(\cf(x,s)).
\end{aligned}
\end{equation*}
\end{proof}

We conclude with an observation about mean curvature flow
of one dimension less.
Suppose that $(\m,g(\tau))$ is a reverse Ricci flow for 
$\tau\in (a,b)$ and that for some $(n-1)$-dimensional 
manifold $P$, the map
$\si: P \times (a,b) \to {\mathcal M}$ is 
a reverse mean curvature flow in the sense that
$P_\tau := \si(P \times \{\tau\})$ is a family of smooth
hypersurfaces of $\m$ and 
$$- \frac{\partial \si}{\partial \tau} (x, \tau) 
= \vec H^{g(\tau)}_{P_\tau}(\si(x, \tau)).$$
Then the space-time track, which is the
image of $\Si:P\times (a,b)\to\hat\m$ defined by
$$\Si(x, \tau):=(\si(x, \tau), \tau),$$
is $\phi$-minimal in $(\hat\m,\hat g)$
for $\phi = \frac{N}{2\tau}$,
up to errors of order $\frac1{N}$,
in the sense that 
$$H_{\phi} := H^{\hat g}_{\Si(P\times (a,b))} 
- \langle\hat \nabla \phi, \nu_{\Si}\rangle = O(N^{-1}),$$ 
where $\nu_{\Si}$ is the unit outward normal to 
$\Si(P \times (a,b))$
and we have used the natural generalization of mean curvature 
introduced in \cite{Gr}. (See also Chapter 11 \S 3.10 of \cite{CLN} 
for the corresponding property of Perelman's metric $\tilde g$
described in Section \ref{intro}.)

{\sc mathematics institute, university of warwick, coventry, CV4 7AL,
uk}\\
\url{http://www.warwick.ac.uk/staff/E.Cabezas-Rivas}\\
\url{http://www.warwick.ac.uk/~maseq}
\end{document}